\documentclass[12pt]{amsart}

\usepackage{amssymb}
\usepackage{enumerate}
\usepackage{eucal}
\usepackage{graphicx}
\usepackage[margin=1in]{geometry}
\usepackage[colorlinks=true, citecolor=blue, linkcolor=blue, urlcolor=black]{hyperref}

\usepackage[T1]{fontenc}
\usepackage[utf8]{inputenc}
\usepackage{hyphenat}

\usepackage{url}

\usepackage[all]{xy} %pacote xypic usado para fazer diagramas comutativos

\newtheorem{theorem}{Theorem}[section]
\newtheorem{corollary}[theorem]{Corollary}

\newtheorem{proposition}[theorem]{Proposition}

\theoremstyle{definition}

\numberwithin{equation}{section}

\title{The triangular numbers are finitely stable}
\address{Instituto Federal de Educação, Ciência e Tecnologia de São Paulo, campus Itaquaquecetuba. IFSP - ITQ: Rua Primeiro de Maio, 500, Itaquaquecetuba - SP, Brazil. CEP: 08571-050}
\email{luan.ferreira@ifsp.edu.br}

\begin{document}

\maketitle

\begin{center} \large{Luan Alberto Ferreira} \end{center}

\begin{abstract}
    This small note proves that the set of triangular numbers is a finitely stable additive basis. This, together with a previous result by the author, shows that triangular numbers and squares are, among all polygonal numbers, the only ones that are finitely stable additive bases.
\end{abstract}

\section{Introduction}

If $A$, $B \subseteq \mathbb{N} = \{0, 1, 2, 3, 4, 5, \ldots\}$, we define \[A + B = \{a + b; \ a \in A, \ b \in B\}.\]

Also, if $h \in \mathbb{N}$, $h \neq 0$, then we write \[hA = \underbrace{A + \cdots + A}_{h \text{ times}} = \{a_1 + \cdots + a_h; \ a_1, \ldots, a_h \in A\}.\]

If there is $h \in \mathbb{N}$, $h \neq 0$, such that $hA = \mathbb{N}$, then $A$ is said to be an additive basis. The smallest $h \in \mathbb{N}$ satisfying this equality is called the order of $A$, and is denoted by $h = o(A)$. For example, Lagrange's four-square theorem says that the set of the squares $\mathbb{N}^2 = \{0, 1, 4, 9, 16, 25, \ldots\}$ is an additive basis whose order is four, i.e., $o(\mathbb{N}^2) = 4$. Finally, if $A \subseteq \mathbb{N}$ and $n \in \mathbb{N}$, then we write \[A(n) = \#\{a \in A; \ 1 \leq a \leq n\}.\]

In \cite{Ferreira}, it was defined that an additive basis $A$ is \emph{finitely stable} if the order of $A$ is equal the order of $A \cup F$, for all finite subset $F \subseteq \mathbb{N}$, and it was proved the following theorem:

\begin{theorem} \label{luan}

Let $A$ be an additive basis such that $o(A) = h \geq 3$. If \[\displaystyle \lim_{n \rightarrow \infty} \frac{((h - 2)A)(n)}{n} = 0\] and \[\limsup_{n \to \infty} \displaystyle \frac{((h - 1)A)(n)}{n} < 1,\] then $A$ is finitely stable.

\end{theorem}

Using this result, it was shown that the set of the squares $\mathbb{N}^2$ is a finitely stable additive basis. But theorem \ref{luan} also implies that the set of triangular numbers $T = \{0, 1, 3, 6, 10, \ldots\}$ is finitely stable. The purpose of this humble note is to prove this and to show that among all polygonal numbers, the triangular numbers and the squares are the only ones that are finitely stable additive bases.

\section{The result}

For the main result we will need the following facts. For completeness, a proof of the second will be given.

\begin{theorem}[Gauss' num $= \Delta + \Delta + \Delta$ theorem]

The set of the triangular numbers $T$ is an additive basis of order $3$. 

\end{theorem}

\begin{proof} See \cite{Nathansonbook}, chapter $1$, section $1.7$. \end{proof}

\begin{proposition}

If $n \equiv 5$ or $8 \pmod 9$, then $n$ is not the sum of two triangular numbers.

\end{proposition}

\begin{proof} Suppose by absurd that $n \equiv 5 \pmod 9$ is the sum of two triangular numbers. Then there exists $m$, $x$, $y \in \mathbb{N}$ such that \[9m + 5 = \frac{x(x+1)}{2} + \frac{y(y+1)}{2}.\] This implies \[18m + 10 = x(x+1) + y(y+1).\] Taking this equation mod $9$, \[1 \equiv x(x+1) + y(y+1) \pmod 9.\] But this equation does not have a solution mod $9$, contradiction. Analogously for the case $n \equiv 8 \pmod 9$. \end{proof}

\begin{corollary}

The triangular numbers $T$ is a finitely stable additive basis.

\end{corollary}

\begin{proof} Since $o(T) = 3$, the previous proposition shows that the hypotheses of theorem \ref{luan} are fulfilled for $T$. Then $T$ is a finitely stable additive basis. \end{proof}

This corollary, together with the next two results, show that the only polygonal numbers that are finitely stable additive bases are the squares $\mathbb{N}^2$ and triangular numbers $T$.

\begin{theorem} [Fermat-Cauchy polygonal number theorem]

Let $k \in \mathbb{N}$, $k \geq 3$. If $\ \mathbb{N}_k$ denotes the set of the $k$-gonal numbers, then $o(\mathbb{N}_k) = k$.
    
\end{theorem}

\begin{proof} See \cite{Nathansonarticle}. \end{proof}

\begin{theorem} [Legendre polygonal number theorem]

Let $m \geq 3$ and $N \geq 28m^3$. If $m$ is odd, then $N$ is the sum of four polygonal numbers of order $m + 2$. If $m$ is even, then $N$ is the sum of five polygonal numbers of order $m + 2$, at least one of which is $0$ or $1$.

\end{theorem}

\begin{proof} See \cite{Nathansonbook}, chapter $1$, section $1.7$. \end{proof}

\begin{corollary}

The only polygonal numbers that are finitely stable additive bases are the squares $\mathbb{N}^2$ and triangular numbers $T$.

\end{corollary}

\begin{proof} Fix $m \in \mathbb{N}$, $m \geq 3$. Now consider the set \[\mathbb{N}_{m + 2} \cup \{j \in \mathbb{N}; \ j < 28m^3\}\] and compare the order of this set with the order of $\mathbb{N}_{m + 2}$. \end{proof}

\end{document}